\documentclass[a4paper]{amsart}
\usepackage{amsmath,amssymb,amsthm,amscd}
\usepackage{enumerate,bm,color,here,url,hyperref}

\usepackage{graphicx}
\graphicspath{{./figure/}}

\makeatletter
\@namedef{subjclassname@2020}{\textup{2020} Mathematics Subject Classification}
\makeatother

\theoremstyle{plain}
	\newtheorem{thm}{Theorem}[section]
	\newtheorem{lem}[thm]{Lemma}
	\newtheorem{prop}[thm]{Proposition}

	\newtheorem{ques}[thm]{Question}
	
\theoremstyle{definition}
	
\theoremstyle{remark}
	\newtheorem{rem}[thm]{Remark}

\newcommand{\fig}[3][width=12cm]{
\begin{figure}[htbp]
	\centering 
	\includegraphics[#1,clip]{#2} 
	\caption{#3} 
	\label{fig:#2}
\end{figure}}

\DeclareMathOperator{\Aut}{Aut}

\DeclareMathOperator{\Inn}{Inn}

\DeclareMathOperator{\Isom}{Isom}

\DeclareMathOperator{\Out}{Out}

\DeclareMathOperator{\vol}{vol}

\newcommand{\cut}{\setminus \! \! \setminus}

\newcommand{\bbZ}{\mathbb{Z}}
\newcommand{\bbR}{\mathbb{R}}

\newcommand{\bbH}{\mathbb{H}}
\newcommand{\bbP}{\mathbb{P}}

\newcommand{\bfi}{\mathbf{i}}
\newcommand{\bfj}{\mathbf{j}}
\newcommand{\bfk}{\mathbf{k}}

\begin{document}

\title[Links in $S^{3} / Q_{8}$ and their lifts]{Links in the spherical 3-manifold obtained from the quaternion group and their lifts}

\author{Ken'ichi Yoshida}
\address{International Institute for Sustainability with Knotted Chiral Meta Matter (WPI-SKCM$^2$), Hiroshima University, 1-3-1 Kagamiyama, Higashi-Hiroshima, Hiroshima 739-8526, Japan}
\email{kncysd@hiroshima-u.ac.jp}

\date{}

\begin{abstract}
We show that there are infinitely many triples of non-isotopic hyperbolic links in the lens space $L(4,1)$ 
such that the three lifts of each triple in $S^{3}$ are isotopic. 
They are obtained as the lifts of links in $S^{3} / Q_{8}$ by double covers, 
where $Q_{8}$ is the quaternion group. 
To construct specific examples, 
we introduce a diagram of a link in $S^{3} / Q_{8}$ obtained by projecting to a square. 
The diagrams of isotopic links are connected by Reidemeister-type moves. 
\end{abstract}

\maketitle

\section{Introduction}
\label{section:intro}

Links in 3-manifolds other than the 3-sphere $S^{3}$ have been investigated using diagrams as knot theory. 
In some cases, diagrams of isotopic links are connected by generalizations of the Reidemeister moves. 
Diagrams often induce invariants such as presentations for the fundamental groups of the complements, polynomial invariants, and skein modules. 
For example, links have been considered in 
the projective 3-space $\bbR \bbP^{3}$ \cite{Drobotukhina90}, 
the lens spaces $L(p,q)$ \cite{BGH08, CMM13}, 
the thickened surfaces $\Sigma \times I$ \cite{Kauffman99}, 
the orientation $I$-bundles over surfaces \cite{Bourgoin08}, 
the 3-torus $T^{3}$ \cite{AEM25, Carrega17, Vuong23}, 
the products of surfaces and the circle $\Sigma \times S^{1}$ \cite{MD09}, 
and more generally, the Seifert fibered spaces \cite{GM16, GM17}.

A link in a 3-manifold $X$ has the lift (i.e. the preimage) in a finite cover $\widetilde{X}$ of $X$. 
Manfredi~\cite{Manfredi14} considered links in lens spaces and their lifts in the 3-sphere $S^{3}$, 
and constructed examples of non-isotopic links in $L(4,1)$ and $L(2q+1,q)$ ($q \geq 2$) having isotopic lifts. 
The lift in $S^{3}$ of a link in a lens space is a freely periodic link in the sense of Hartley \cite{Hartley81}, 
which is a type of symmetric link in $S^{3}$. 
The symmetries of knots (i.e. one-component links) in $S^{3}$ are quite restricted. 
For instance, a prime knot in $S^{3}$ has a unique freely periodic symmetry \cite{BF87, Sakuma86}. 
Moreover, several constraints for freely periodic knots in $S^{3}$ are known \cite{BR23, Chbili03, Hartley81, Nozaki18}. 
More generally, the finite group actions on knots in $S^{3}$ are classified \cite{BRW23}.

A (singly, doubly, or triply) periodic tangle is defined as the lift with respect to the universal covering of a link in the solid torus $S^{1} \times D^{2}$ \cite{SZ25}, the thickened torus $T^{2} \times I$ \cite{GMO07}, or the 3-torus $T^{3}$ \cite{AEM25}, 
and used to model systems of entangled chains such as polymers, crystalline materials, and textiles. 
In relation to this, Kotorii, Mahmoudi, Matsumoto, and the author \cite{KMMY25} showed that 
if two links in $X = S^{1} \times D^{2}$, $T^{2} \times I$, or $T^{3}$ have isotopic lifts in a finite cover of $X$, 
then they are isotopic. 
Similarly, the author \cite{Yoshida25} showed that 
if two links in $\bbR \bbP^{3}$ have isotopic lifts in $S^{3}$, then they are isotopic.

In this paper, we consider links in the 3-manifold $S^{3} / Q_{8}$ and their lifts in covers, 
where $Q_{8} = \{ \pm 1, \pm \bfi, \pm \bfj, \pm \bfk \in \bbH \}$ is the quaternion group. 
The 3-manifold $S^{3} / Q_{8}$ is topologically obtained from a cube by gluing each pair of opposite faces using the counterclockwise $\pi/2$-rotation. 
We obtain a diagram of a link in $S^{3} / Q_{8}$ by projecting this cube to a square. 
The lens space $L(4,1)$ is a double cover of the 3-manifold $S^{3} / Q_{8}$ in three ways. 
The three lifts in $L(4,1)$ of a link in $S^{3} / Q_{8}$ are different in general, 
but they have a common lift in $S^{3}$. 
In Theorem~\ref{thm:main}, we will show that there are infinitely many triples of non-isotopic hyperbolic links in $L(4,1)$ such that the three lifts of each triple in $S^{3}$ are isotopic. 

In Section~\ref{section:diagram}, 
we will introduce diagrams of links in $S^{3} / Q_{8}$ as well as Reidemister-type moves. 
In Section~\ref{section:cover}, 
we will consider lifts of links in $S^{3} / Q_{8}$. 
To prove Theorem~\ref{thm:main}, 
we will show that there are sufficiently complicated hyperbolic links in a 3-manifold 
in Theorem~\ref{thm:many-hyp}. 
We will also describe how to construct diagrams of lifts of links in $S^{3} / Q_{8}$. 
In Section~\ref{section:ex}, 
we will give three examples of lifts of links in $S^{3} / Q_{8}$ using diagrams. 
In the first example, the links consist of fibers of Seifert fibrations, 
and we can reconstruct examples in \cite{Manfredi14}. 
In the second and last examples, the links are hyperbolic. 
In the last example, the three links in $L(4,1)$ have the same number of components.

\section{Diagrams of links in $S^{3} / Q_{8}$}
\label{section:diagram}

The quaternion algebra $\bbH = \{ a + b \bfi + c \bfj + d \bfk \mid a, b, c, d \in \bbR \}$ is 
the skew field with multiplication defined by $\bfi^{2} = \bfj^{2} = \bfk^{2} = \bfi \bfj \bfk = -1$. 
The \emph{quaternion group} $Q_{8}$ consists of the eight elements $\pm 1, \pm \bfi, \pm \bfj, \pm \bfk \in \bbH$ with the multiplication of quaternions. 
We regard the 3-sphere $S^{3}$ as the set of unit quaternions. 
We consider the right action of the group $Q_{8}$ on $S^{3}$ by multiplication. 
The action is free and isometric with respect to the spherical metric on $S^{3}$. 
The quotient space $S^{3} / Q_{8}$ is the simplest one of prism 3-manifolds, 
which are spherical 3-manifolds of dihedral type. 
Then each double cover of $S^{3} / Q_{8}$ is the lens space $L(4,1)$, not $L(4,3)$, 
for the naturally induced orientation, 
as it will be shown in Section~\ref{section:cover}. 

We consider the Dirichlet fundamental domain $C \subset S^{3}$ for $Q_{8}$ centered at $1 \in S^{3}$. 
The domain $C$ is bounded by the perpendicular bisector planes of the six segments from 1 to $ \pm \bfi, \pm \bfj, \pm \bfk$. 
Hence $C$ is combinatorially a cube. 
The images of $C$ by the action of $Q_{8}$ are the facets of a 4-cube (tesseract). 
Consequently, the 3-manifold $S^{3} / Q_{8}$ is obtained from a cube 
by gluing each pair of opposite faces using the counterclockwise $\pi/2$-rotation. 
The three arrows in Figure~\ref{fig:cube-domain.pdf} are glued together. 
Note that the action of $\bfi \in Q_{8}$ maps $1, \bfj, \bfk$ to $\bfi, -\bfk, \bfj$, respectively. 

\fig[width=4cm]{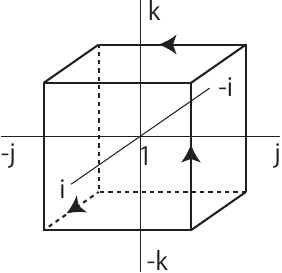}{The fundamental domain $C$}

A \emph{link} $L$ in a 3-manifold $X$ is a collection of finitely many disjoint circles embedded in the interior of $X$. 
Two links $L_{0}$ and $L_{1}$ are \emph{isotopic} 
if there exists an ambient isotopy $F \colon X \times [0,1] \to X$ 
such that $F(\cdot, 0)$ is the identity map on $X$ and $F(L_{0}, 1) = L_{1}$. 
We consider non-oriented smooth (or PL) links in $S^{3} / Q_{8}$. 
Let us construct a diagram of a link $L$ in $S^{3} / Q_{8}$ 
by projecting the cube $C$, regarded as a regular cube in the Euclidean space, to a square. 
We set the link $L$ in general position with respect to the decomposition of $S^{3} / Q_{8}$ into $C$. 
In other words, $L$ is disjoint from the edges and intersects the faces transversely. 
Then $L$ is obtained from $L' \subset C$ by the gluing, 
where $L'$ is a disjoint union of properly embedded arcs and circles. 
Fix an axis for $C$ (in the direction of $\bfi$ for example). 
We obtain a diagram of $L$ from $L'$ by the orthogonal projection of the cube $C$ to a square $S$ 
(in the direction perpendicular to $\bfi$ for example). 
Here we need to assume that $L'$ is in general position with respect to the projection. 
We call a point of $L' \cap \partial C$ (or its image in the square $S$) a \emph{boundary point}. 
Each boundary point is respectively called \emph{horizontal}, \emph{vertical}, and \emph{internal} 
if it is contained in the horizontal edges, the vertical edges, and the interior of $S$. 
To determine the link $L$ by a diagram, we label the boundary points. 
Each pair of horizontal or vertical boundary points which are glued has the same label. 
The heights of horizontal and vertical boundary points in $C$ are determined by the positions of opposite boundary points. 
The internal boundary points in the upper face of $C$ have positive labels 
and those in the lower face have negative labels. 
Each pair of internal boundary points labeled with $+t$ and $-t$ are glued. 

In summary, a diagram of a link $L$ in $S^{3} / Q_{8}$ consists of immersed arcs and circles on $S$ that satisfies the following conditions (called the \emph{diagram conditions}): 
\begin{enumerate}
\item The intersection (crossings) is transversal. 
\item The crossings are finitely many double points. 
\item The crossings are contained in the interior of $S$. 
\item The boundary points are disjoint from the vertices of $S$. 
\item Each crossing has data of the overpass and underpass. 
\item The horizontal (resp. vertical) boundary points are labeled with $1, \dots, n_{h}$ (resp. $1, \dots, n_{v}$). 
\item The internal boundary points (depicted by dots) are labeled with $\pm 1, \dots, \pm n_{i}$. 
\item The internal boundary point labeled with $t \in \{ +1, \dots, +n_{i} \}$ is mapped to that labeled with $-t$ by the counterclockwise $\pi/2$-rotation on $S$. 
\end{enumerate}
The conditions (1)--(4) concern genericity. 
The conditions (5)--(8) concern reconstruction of a link. 
By gluing the corresponding boundary points,
we can determine the link $L$ in $S^{3} / Q_{8}$ from a diagram (see Figure~\ref{fig:square-diagram.pdf}). 

\fig[width=8cm]{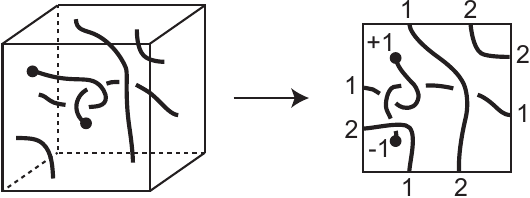}{A diagram of a link in $S^{3} / Q_{8}$}

We consider moves connecting different diagrams of isotopic links in $S^{3} / Q_{8}$. 
Our arguments are similar to those in \cite{CMM13} for a link in a lens space. 
A \emph{diagram isotopy} is an isotopy on the square $S$ 
that keeps the diagram conditions (1)--(8) and possibly permutes the labels. 
The \emph{generalized Reidemeister moves} on a diagram of a link in $S^{3} / Q_{8}$ 
are the moves $R_{1}, \dots, R_{8}$ shown in Figure~\ref{fig:reidemeister.pdf}. 
The moves $R_{1}$, $R_{2}$, and $R_{3}$ correspond to the ordinary Reidemeister moves. 
Of course, the crossing in $R_{1}$ may be inverse. 
The moves $R_{1}, \dots, R_{5}$ occur in the interior of $S$. 
The moves $R_{6}$, $R_{7}$, and $R_{8}$ concern edges of $S$, 
and those obtained by rotations of $S$ are also the moves of the same type. 
The labels on an edge in $R_{6}$ may be inverse thanks to $R_{7}$. 
One of the moves $R_{7}$ can be expressed by a composition of the other of $R_{7}$ and $R_{2}$. 
Two of the moves $R_{8}$ can be expressed by compositions of the remaining of $R_{8}$ and $R_{6}$. 
For example, the second $R_{8}$ is a composition of the first $R_{8}$ and $R_{6}$ as shown in Figure~\ref{fig:R6-R8.pdf}.

\fig[width=12cm]{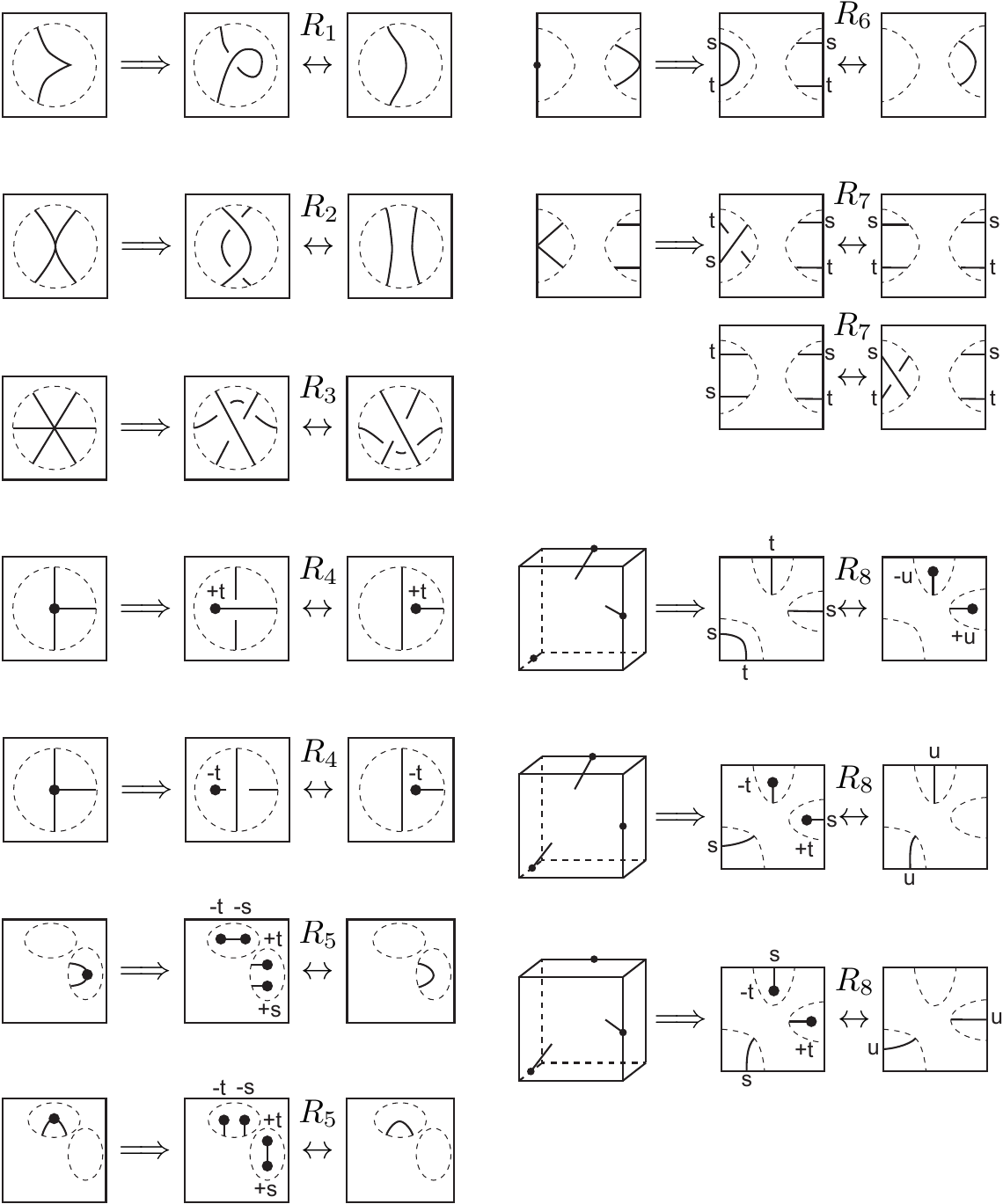}{Generalized Reidemeister moves}

\fig[width=12cm]{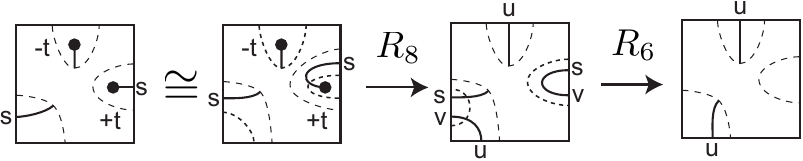}{An $R_{8}$ move as a composition of $R_{8}$ and $R_{6}$ moves}

\begin{thm}
\label{thm:move}
Two links $L_{0}$ and $L_{1}$ in $S^{3} / Q_{8}$ are isotopic 
if and only if their diagrams are connected by a finite sequence of diagram isotopies and generalized Reidemeister moves $R_{1}, \dots, R_{8}$. 
\end{thm}
\begin{proof}
It is easy to see that each generalized Reidemeister move connects isotopic links. 
Hence a finite sequence of diagram isotopies and generalized Reidemeister moves does not change the isotopy class of the link. 

Conversely, suppose that two links $L_{0}$ and $L_{1}$ in $S^{3} / Q_{8}$ are isotopic. 
There exists an ambient isotopy connecting $L_{0}$ and $L_{1}$ 
in general position for the square $S$. 
Then the intermediate links in the isotopy may violate the genericity condition finitely many times. 
Their projections contain a finite number of 1-codimensional forbidden configurations shown in Figure~\ref{fig:reidemeister.pdf}. 
The forbidden configurations correspond to the moves $R_{1}, \dots, R_{8}$. 
More precisely, violation of genericity induces the moves as follows: 
\begin{itemize}
\item Violation of transversality of the intersection of $L$ and the faces of $C$ induces $R_{5}$ and $R_{6}$. 
\item Violation of immersion induces $R_{1}$. 
\item Violation of the diagram condition (1) induces $R_{2}$ and $R_{4}$. 
\item Violation of the diagram condition (2) induces $R_{3}$. 
\item Violation of the diagram condition (3) induces $R_{7}$. 
\item Violation of the diagram condition (4) and disjointness of $L$ from the edges of $C$ induce  $R_{8}$. 
\end{itemize}
The remaining part of the isotopy is realized by diagram isotopies. 
\end{proof}

\begin{rem}
The diagram condition (8) might be inconvenient 
since it is not truly combinatorial. 
If we want to treat a diagram more combinatorially, 
we need another constraint for the internal boundary points instead of the condition (8). 
For example, 
we may consider a diagram of another type 
in which the boundary points are drawn in the edges, 
in the same manner as a disk diagram of a link in a lens space explained in Section~\ref{section:cover}. 
Then we need other moves for crossings of internal boundary points. 
\end{rem}

We describe self-homeomorphisms of $S^{3} / Q_{8}$. 
For a manifold $X$, let $\Aut(X)$ denote the group consisting of self-homeomorphisms of $X$. 
Let $\Aut_{0}(X)$ denote the path-connected component of the identity in $\Aut(X)$ with respect to the compact-open topology. 
The quotient group $\Aut(X) / \Aut_{0}(X) = \pi_{0}(\Aut(X))$ is called the \emph{mapping class group} of $X$. 
Isotopic self-homeomorphisms are in the same class. 
For a group $G$, the quotient group $\Out(G) = \Aut(G) / \Inn(G)$ 
is called the \emph{outer automorphism group} of $G$, 
where $\Aut(G)$ is the group consisting of automorphisms of $G$, 
and $\Inn(G)$ is the group consisting of automorphisms of $G$ by conjugations. 
The action of self-homeomorphisms on the fundamental group induces the natural homomorphism 
$\pi_{0}(\Aut(X)) \to \Out(\pi_{1}(X))$. 

Price~\cite{Price77} proved that the natural homomorphism $\pi_{0} (\Aut(S^{3} / Q_{8})) \to \Out(Q_{8})$ is an isomorphism. 
Furthermore, the group $\Out(Q_{8})$ is isomorphic to the symmetric group $\mathfrak{S}_{3}$ of degree 3. 
Indeed, the group $\Out(Q_{8})$ consists of the six classes of maps that assign 
$(\bfi, \bfj, \bfk)$ to $(\bfi, \bfj, \bfk), (\bfj, \bfk, \bfi), (\bfk, \bfi, \bfj), (\bfi, \bfk, -\bfj), (-\bfk, \bfj, \bfi), (\bfj, -\bfi, \bfk)$. 
They come from the symmetry of the cube $C$. 
There does not exist an orientation-reversing self-homeomorphism of $S^{3} / Q_{8}$. 
The $\pi$-rotation of $C$ around an axis orthogonal to faces represents the trivial element of $\Out(Q_{8})$. 
Hence the corresponding self-homeomorphism of $S^{3} / Q_{8}$ is isotopic to the identity. 
This fact for the axis $\bfi$ follows from $-\bfi \bfi \bfi = \bfi$, $-\bfi \bfj \bfi = -\bfj$, and $-\bfi \bfk \bfi = -\bfk$. 

For a link $L$ in $S^{3} / Q_{8}$, 
there are six (possibly coinciding) diagrams of $L$ depending on the choice of directions of $\bfi$, $\bfj$, and $\bfk$. 
Note that two diagrams related by the $\pi/2$-rotation of the square $S$ 
are not equivalent under the generalized Reidemeister moves in general, 
but two diagrams related by the $\pi$-rotation of $S$ are equivalent. 
Since any self-homeomorphism of $S^{3} / Q_{8}$ is isotopic to a self-homeomorphism $f$ 
obtained by a rotation of $C$. 
Then the six diagrams of the link $f(L)$ coincide with those of $L$ up to the permutation of $\bfi$, $\bfj$, and $\bfk$.

\section{Coverings}
\label{section:cover}

In this section, we consider the lifts (i.e. the preimages) of links in $X = S^{3} / Q_{8}$ by the covers of $X$. 
The proper subgroups of the quaternion group $Q_{8} = \{ \pm 1, \pm \bfi, \pm \bfj, \pm \bfk \}$ 
are $\{ 1 \}$, $\{ \pm 1 \}$, $\{ \pm 1, \pm \bfi \}$, $\{ \pm 1, \pm \bfj \}$, and $\{ \pm 1, \pm \bfk \}$. 
Let $X_{\bfi}$, $X_{\bfj}$, and $X_{\bfk}$ denote the covers of $X$ corresponding to 
the subgroups $\{ \pm 1, \pm \bfi \}$, $\{ \pm 1, \pm \bfj \}$, and $\{ \pm 1, \pm \bfk \}$, respectively. 
The 3-manifolds $X_{\bfi}$, $X_{\bfj}$, and $X_{\bfk}$ are homeomorphic to the lens space $L(4,1)$. 
We will describe them explicitly later. 
For a link $L$ in $X$, 
let $L_{\bfi}$, $L_{\bfj}$, and $L_{\bfk}$ denote the lifts of $L$ in $X_{\bfi}$, $X_{\bfj}$, and $X_{\bfk}$ 
by the covering maps. 
In fact, as it will be shown in Proposition~\ref{prop:reverse}, 
the isotopy classes of the links $L_{\bfi}$, $L_{\bfj}$, and $L_{\bfk}$ in $L(4,1)$ do not depend on the choices of homeomorphisms from $X_{\bfi}$, $X_{\bfj}$, and $X_{\bfk}$ to $L(4,1)$. 
The links $L_{\bfi}$, $L_{\bfj}$, and $L_{\bfk}$ have a common lift $\widetilde{L}$ in $S^{3}$, 
which is the lift of $L$. 
However, the links $L_{\bfi}$, $L_{\bfj}$, and $L_{\bfk}$ in $L(4,1)$ are not isotopic in general. 
Remark that the links $L_{\bfi}$, $L_{\bfj}$, and $L_{\bfk}$ in $L(4,1)$ have a common lift also in the projective 3-space $\bbR \bbP^{3}$, 
since the covering map from $\bbR \bbP^{3}$ to $X$ is unique (it corresponds to the subgroup $\{ \pm 1 \} \subset Q_{8}$). 

We discuss this more precisely. 
The easiest way to distinguish links is to compare the numbers of components. 
The free homotopy classes of closed curves in $X$ 
correspond to the conjugacy classes in $\pi_{1}(X)$: 
$[1]$, $[-1]$, $[\bfi] = [-\bfi]$, $[\bfj] = [-\bfj]$, and $[\bfk] = [-\bfk]$. 
For a knot $K$ in $X$, the number of components of lifts $K_{\bfi}$ in $X_{\bfi}$ is determined by its free homotopy class. 
If the free homotopy class of $K$ corresponds to $[1]$, $[-1]$, or $[\bfi]$, 
then the number of components of $K_{\bfi}$ is equal to two. 
Otherwise, it is equal to one. 

A link $L$ in a closed 3-manifold $M$ is \emph{hyperbolic} 
if the complement $M \setminus L$ admits a finite volume hyperbolic structure. 
Sufficiently complicated links tend to be hyperbolic and have large volume. 
We can obtain such links with the same free homotopy classes of components. 
Indeed, the following theorem holds. 

\begin{thm}
\label{thm:many-hyp}
Let $M$ be a 3-manifold with possibly empty boundary consisting of tori. 
Fix any $V > 0$ and free homotopy classes $c_{1}, \dots, c_{n}$ of close curves in $M$. 
Then there is an $n$-component hyperbolic link $L$ in $M$ 
such that $\vol (M \setminus L) > V$ and the free homotopy classes of components are $c_{1}, \dots, c_{n}$. 
Consequently, there are infinitely many such hyperbolic links. 
\end{thm}
\begin{proof}
Let $L_{0}$ be an $n$-component link in $M$ 
such that the free homotopy classes of components are $c_{1}, \dots, c_{n}$. 
For a link or an embedded graph $\Gamma$ in $M$, 
the complement of a open regular neighborhood of $\Gamma$ in $M$ 
is called the \emph{exterior} of $\Gamma$ in $M$ 
and denoted by $M \cut \Gamma$. 
By applying a result of Myers \cite[Theorem 1.1]{Myers93} to the exterior $M \cut L_{0}$, 
we obtain an embedded arc $\gamma$ in the interior of $M$ that satisfies the following: 
\begin{enumerate}
\item $\gamma \cap L_{0}$ consists of the endpoints $p_{0}$ and $p_{1}$ of $\gamma$. 
\item $p_{0}$ and $p_{1}$ are contained in a single component $K$ of $L_{0}$. 
\item The exterior $M \cut (\gamma \cup L_{0})$ contains no essential surfaces of non-negative Euler characteristic. 
\end{enumerate}
The boundary components of $M \cut (\gamma \cup L_{0})$ consist of 
$n-1$ tori $T_{1}, \dots, T_{n-1}$ and a closed surface $\Sigma$ of genus two. 
Due to Thurston's hyperbolization \cite{Thurston82}, the condition (3) implies that 
$M_{0} = (M \cut (\gamma \cup L_{0})) \setminus (T_{1} \cup \dots \cup T_{n-1})$ admits a finite volume hyperbolic structure with totally geodesic boundary $\Sigma$. 
Note that 
a 3-manifold with boundary consisting of surfaces of negative Euler characteristic 
admits a finite volume hyperbolic structure with totally geodesic boundary 
if and only if its double admits a finite volume hyperbolic structure 
(see \cite{AST07, CFW10} for more details).

Let $N(\gamma)$ denote a small closed regular neighborhood of $\gamma$. 
Then the 3-ball $N(\gamma)$ intersects $K$ at two arcs $\beta_{0}$ and $\beta_{1}$. 
The set $K \setminus N(\gamma)$ is the union of two arcs $\alpha_{0}$ and $\alpha_{1}$. 
There are two annuli $A_{0}$ and $A_{1}$ in $\Sigma$ that are respectively contained in the boundary of a closed regular neighborhood of $\alpha_{0}$ and $\alpha_{1}$. 
Then $M_{0}' = M_{0} \setminus (A_{0} \cup A_{1})$ admits a finite volume hyperbolic structure with totally geodesic boundary $\Sigma' = \Sigma \setminus (A_{0} \cup A_{1})$. 
Here $\Sigma'$ is a 4-punctured sphere, and $A_{0}$ and $A_{1}$ are rank-1 cusps. 
This follows from the fact that 
the manifold $M_{0}'$ still satisfies the acylindricality for Thurston's hyperbolization. 
In another way, we may consider the double of $M_{0}$ along the totally geodesic boundary $\Sigma$. 
Then we obtain the doubles of $A_{0}$ and $A_{1}$ as cusps of the double of $M_{0}'$ 
by drilling the two corresponding closed geodesics in the image of $\Sigma$, derived from the meridian of $K$. 
More generally, the complement of a simple geodesic link in a hyperbolic 3-manifold is also hyperbolic \cite[Proposition 4]{Kojima88}. 

We prepare hyperbolic tangles of large volume. 
For $n \geq 1$, let $T_{n}$ be a tangle in a 3-ball $B$ that consists two arcs $\tau_{0}$ and $\tau_{1}$ 
shown in the left of Figure~\ref{fig:tangle_n.pdf}. 
In Lemma~\ref{lem:tangle}, we will show that $T_{n}$ is hyperbolic and $\vol(T_{n}) > V$ for sufficiently large $n$. 
Here we say $T_{n}$ is hyperbolic 
if the complement of $T_{n}$ admits a hyperbolic structure 
such that the 4-punctured sphere $\partial B \setminus (\tau_{0} \cup \tau_{1})$ is totally geodesic, 
and $\tau_{0}$ and $\tau_{1}$ correspond to rank-1 cusps. 
We fix such $n$. 

We obtain a link $L$ by replacing $N(\gamma)$ with the tangle $T_{n}$ 
so that the endpoints of each $\tau_{i}$ for $i = 0,1$ are glued to the endpoints of $\beta_{i}$. 
Then the link $L$ is homotopic to the link $L_{0}$. 
The manifold $M \setminus L$ is obtained by gluing the manifolds $M_{0}'$ and $T_{n}$ along the totally geodesic boundary. 
Note that the hyperbolic structures on these 4-punctured spheres do not necessarily coincide. 
Then the manifold $M \setminus L$ admits a finite volume hyperbolic structure, and $\vol (M \setminus L) \geq \vol(M_{0}') + \vol(T_{n}) > V$ by Theorem~\ref{thm:glue}. 
\end{proof}

\begin{thm}
\label{thm:glue}
For each $i = 1,2$, let $M_{i}$ be a finite volume hyperbolic 3-manifold with totally geodesic boundary $\Sigma_{i}$. 
Suppose that there is a homeomorphim $f \colon \Sigma_{1} \to \Sigma_{2}$. 
Let $M_{1} \cup_{f} M_{2}$ denote the manifold obtained by gluing $M_{1}$ and $M_{2}$ along the boundary $\Sigma_{1}$ and $\Sigma_{2}$ via $f$. 
Then $M_{1} \cup_{f} M_{2}$ admits a finite volume hyperbolic structure satisfying 
\[
\vol (M_{1} \cup_{f} M_{2}) \geq \vol (M_{1}) + \vol (M_{2}). 
\] 
\end{thm}

Hyperbolicity of $M_{1} \cup_{f} M_{2}$ was shown in \cite[Theorem 5.5]{CFW10}. 
Indeed, $M_{1} \cup_{f} M_{2}$ is atoroidal. 
For otherwise an essential torus in $M_{1} \cup_{f} M_{2}$ 
induces an essential torus or annulus in $M_{1}$ or $M_{2}$, 
which contradicts the fact that $M_{1}$ and $M_{2}$ are atoroidal and acylindrical. 
The volume estimate directly follows from \cite[Theorem 9.1]{AST07}, 
where it was shown in a more general setting. 
See also \cite[Theorem 1.5]{Adams21}. 

\fig[width=9cm]{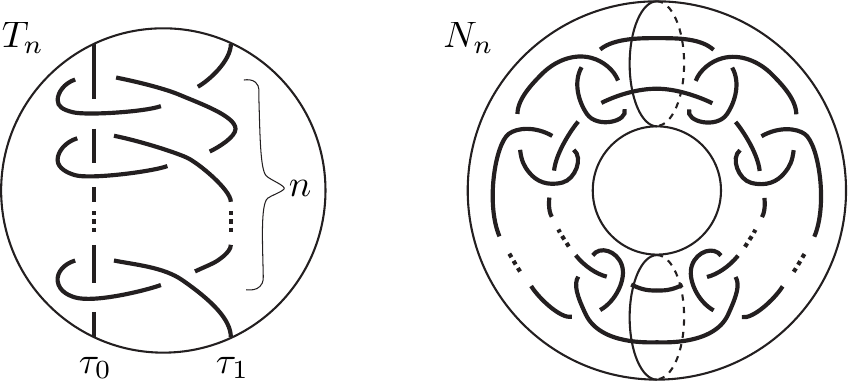}{The tangle $T_{n}$ and the manifold $N_{n}$}

\fig[width=10cm]{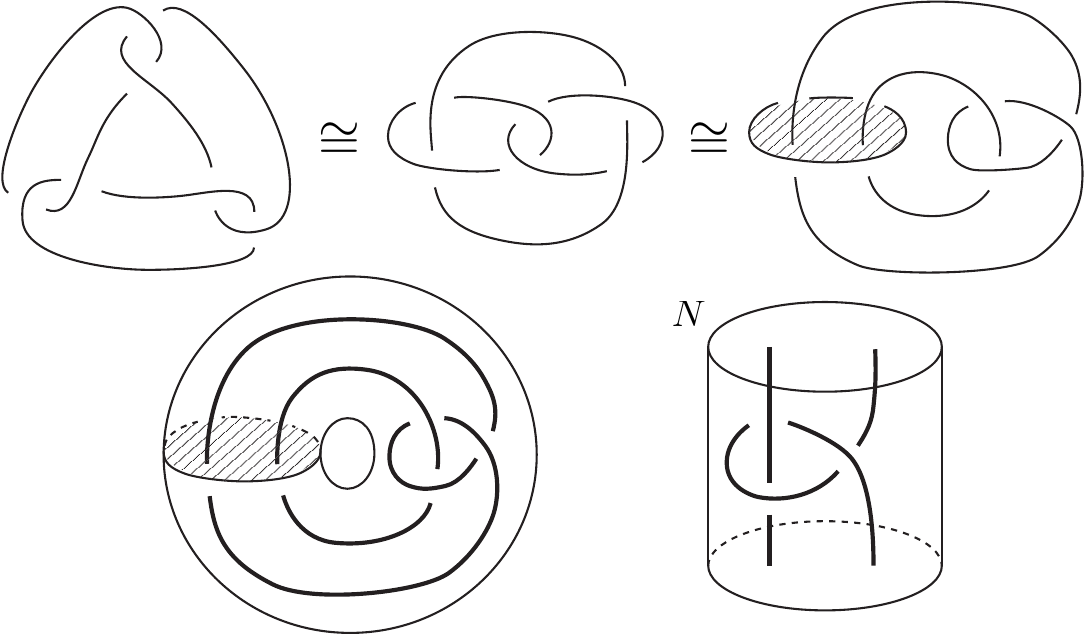}{The $6^{3}_{1}$ link and the manifold $N$}

\begin{lem}
\label{lem:tangle}
Fix $V > 0$. 
Then the tangle $T_{n}$ is hyperbolic and $\vol(T_{n}) > V$
for sufficiently large $n$. 
\end{lem}
\begin{proof}
We show that the double $DT_{n}$ of $T_{n}$ along the boundary $\partial B$ is hyperbolic for large $n$. 
By abuse of notation, we continue to write $T_{n}$ and $DT_{n}$ for their complements. 
We begin with the $6^{3}_{1}$ link in $S^{3}$ shown in the first line of Figure~\ref{fig:magic.pdf}. 
Its complement is called the magic 3-manifold, 
and the hyperbolic structure was constructed in \cite[Example 6.8.2]{Thurston78}. 
We use the third diagram in Figure~\ref{fig:magic.pdf} to show the decomposition along a 3-punctured sphere. 
In the lower left of Figure~\ref{fig:magic.pdf}, we show a link in the solid torus whose complement is homeomorphic to the $6^{3}_{1}$ link complement. 
An essential 3-punctured sphere is isotopic to a totally geodesic one due to Adams \cite{Adams85}. 
By cutting the magic 3-manifold along the totally geodesic 3-punctured sphere (2-punctured disk) of the shaded area, 
we obtain the hyperbolic 3-manifold $N$ with totally geodesic boundary and three rank-1 cusps. 
We construct a hyperbolic 3-manifold $N_{n}$ by gluing $2n$ copies of $N$ and its mirror image along the totally geodesic boundary as shown in the right of Figure~\ref{fig:tangle_n.pdf}.
Furthermore, we obtain the manifold $DT_{n}$ by a Dehn filling of $N_{n}$. 
Then we can apply a result of Futer, Kalfagianni, and Purcell \cite[Theorem 1.1]{FKP08}. 
Fix a horocusp neighborhood $C$ of the cusp of $N$ that is drawn outermost. 
By gluing copies of $C$, we obtain a horocusp neighborhood $C_{n}$ in $N_{n}$. 
Let $\ell_{n}$ denote the length of the slope $s_{n}$ on $\partial C_{n}$ 
such that the Dehn filling along $s_{n}$ gives $DT_{n}$. 
Here the length is measured for a simple closed geodesic in a Euclidean torus. 
If $\ell_{n} > 2\pi$, then $DT_{n}$ is hyperbolic, and 
\[
\vol(DT_{n}) \geq \left( 1 - \left( \frac{2\pi}{\ell_{n}} \right)^{2} \right)^{3/2} \vol(N_{n}). 
\]
Since the manifold $N_{n}$ can be constructed as a double, 
the slope $s_{n}$ is orthogonal to slope corresponding to the core of $C$. 
Hence $\ell_{n} = n \ell$ for some $\ell > 0$.
Moreover, $\vol(N_{n}) = 2n \vol(N)$. 
Therefore the manifold $DT_{n}$ is hyperbolic and $\vol(DT_{n}) > 2V$ for large $n$. 
Then $T_{n}$ is hyperbolic and $\vol(T_{n}) > V$. 
\end{proof}

Now we can show the main theorem. 

\begin{thm}  
\label{thm:main}
There are infinitely many triples of non-isotopic hyperbolic links in $L(4,1)$ 
such that the three lifts of each triple in $S^{3}$ are isotopic. 
\end{thm}
\begin{proof}
Let $L \subset X$ be a link consisting of 
$n_{1}$ components of the free homotopy class $[\bfi]$, 
$n_{2}$ components of $[\bfj]$, and 
$n_{3}$ components of $[\bfk]$. 
Suppose that $n_{1}$, $n_{2}$, and $n_{3}$ are mutually distinct. 
Then the links $L_{\bfi} \subset X_{\bfi}$, $L_{\bfj} \subset X_{\bfj}$, and $L_{\bfk} \subset X_{\bfk}$ 
respectively have $2n_{1} + n_{2} + n_{3}$, $n_{1} + 2n_{2} + n_{3}$, and $n_{1} + n_{2} + 2n_{3}$ components. 
Since these three numbers are distinct, 
the three links $L_{\bfi}$, $L_{\bfj}$, and $L_{\bfk}$ in $L(4,1)$ are distinct. 
However, the lifts of $L_{\bfi}$, $L_{\bfj}$, and $L_{\bfk}$ in $S^{3}$ are the same as the lift of $L$. 
Theorem~\ref{thm:many-hyp} implies that there are infinitely many choices of such a hyperbolic link $L \subset X$. 
\end{proof}

\begin{rem}
Suppose that $L_{\bfi}$, $L_{\bfj}$, and $L_{\bfk}$ are distinct hyperbolic links in $L(4,1)$ 
whose lifts in $S^{3}$ are isotopic to a link $\widetilde{L}$ as above. 
Then deck transformations of the covering map $S^{3} \to L(4,1)$ induce 
three $\bbZ / 4\bbZ$-actions on the hyperbolic 3-manifold $S^{3} \setminus \widetilde{L}$. 
The actions are isotopic to isometric actions by the Mostow rigidity. 
These three isometric $\bbZ / 4\bbZ$-actions are different, 
since $L_{\bfi}$, $L_{\bfj}$, and $L_{\bfk}$ are different. 
\end{rem}

\begin{rem}
The lift in $S^{3}$ of a link in $S^{3} / Q_{8}$ has at least two components. 
Due to Boileau and Flapan~\cite{BF87} and Sakuma~\cite{Sakuma86}, 
a prime knot in $S^{3}$ has a unique freely periodic symmetry, 
and so it cannot be a common lift of mutually non-isotopic knots in $L(4,1)$. 
\end{rem}

It is to be expected that the above construction can be applied to other 3-manifolds. 
We may ask whether there are examples beyond the above construction. 

\begin{ques}
Suppose that links $L_{0}$ and $L_{1}$ in a 3-manifold $X$ have isotopic lifts $\widetilde{L}$ in a finite cover $\widetilde{X}$ of $X$. 
Then is there a link $\widehat{L}$ in a 3-manifold (or orbifold) $\widehat{X}$ 
such that $L_{0}$ and $L_{1}$ are the lifts of $\widehat{L}$ by two covering maps from $X$ to $\widehat{X}$? 
\end{ques}

This assertion is true for hyperbolic links. 
Indeed, the deck transformations induce two subgroups $G_{0}$ and $G_{1}$ of the finite group $\Isom (\widetilde{X} \setminus \widetilde{L})$. 
Let $G$ denote the subgroup of $\Isom (\widetilde{X} \setminus \widetilde{L})$ generated by $G_{0}$ and $G_{1}$. 
Since the action of $G$ on $\widetilde{X} \setminus \widetilde{L}$ preserves the meridians, 
it extends to an action on $\widetilde{X}$. 
Then the orbifold $\widehat{X} = \widetilde{X} / G$ and the ``link'' $\widehat{L} = \widetilde{L} / G$ are desired. 
Note that $\widehat{L}$ might intersect the singular set of $\widehat{X}$. 

We construct diagrams of lifts of a link in $S^{3} / Q_{8}$. 
A \emph{disk diagram} of a link in a lens space was introduced in \cite{CMM13} as follows. 
Let $p$ and $q$ be positive coprime integers. 
The lens space $L(p,q)$ is obtained from a lens-shaped fundamental domain $B$ by gluing the opposite faces using the counterclockwise $2\pi q/ p$-rotation. 
Let $L$ be a link in $L(p,q)$ located in general position for the fundamental domain $B$. 
Then $L$ is obtained from $L' \subset B$ by the gluing, 
where $L'$ is a disjoint union of properly embedded arcs and circles. 
We may isotope $L$ so that the endpoints of $L'$ are contained in small neighborhood of the equator of $B$. 
We obtain a diagram of $L$ from $L'$ by the projection of $B$ to a disk, 
where the equator of $B$ is projected to the boundary of the disk. 
The endpoints of arcs (called the \emph{boundary points}) are drawn in the boundary of the disk (see Figure~\ref{fig:disk-diagram.pdf}). 

\fig[width=9cm]{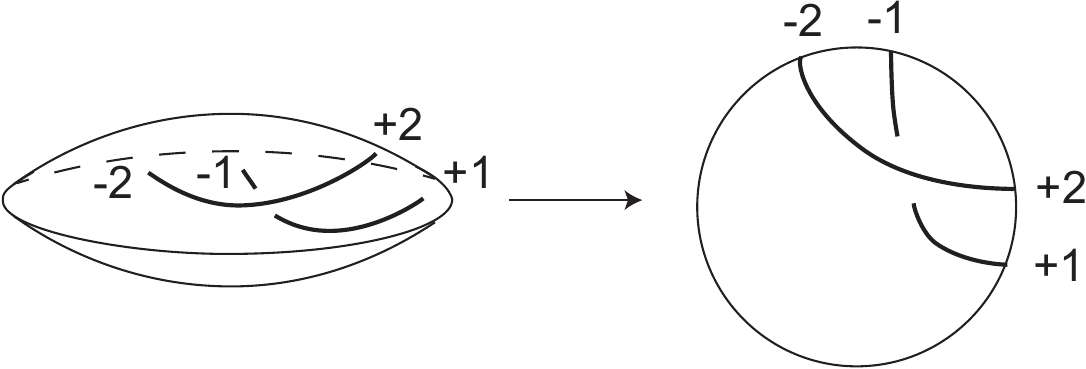}{A diagram of a link in a lens space}

A disk diagram consists of properly immersed arcs and circles on a disk with data of the overpass and underpass at each crossing and labels $\pm 1, \dots, \pm n$ on the boundary points. 
The boundary points labeled with $+t$ and $-t$ come from the endpoint of $L'$ respectively in the upper and lower faces, which are glued together in $L$. 
By gluing the corresponding boundary points, 
we can determine the link $L$ in $L(p,q)$ from a disk diagram. 
The diagrams of isotopic links are connected by generalized Reidemeister moves 
introduced by Cattabriga, Manfredi, and Mulazzani \cite{CMM13}. 
The generalized Reidemeister moves for disk diagrams of links in $L(4,1)$ 
are illustrated in Figure~\ref{fig:disk-reidemeister.pdf} 
although we omit details. 
A disk diagram is called \emph{standard} 
if the labels of the boundary points are arranged in the order of $(+1, \dots, +n, -1, \dots, -n)$ on the boundary. 
For any link in a lens space, we can obtain a standard disk diagram by $R_{6}$ moves \cite[Proposition 1]{Manfredi14}. 

\fig[width=12cm]{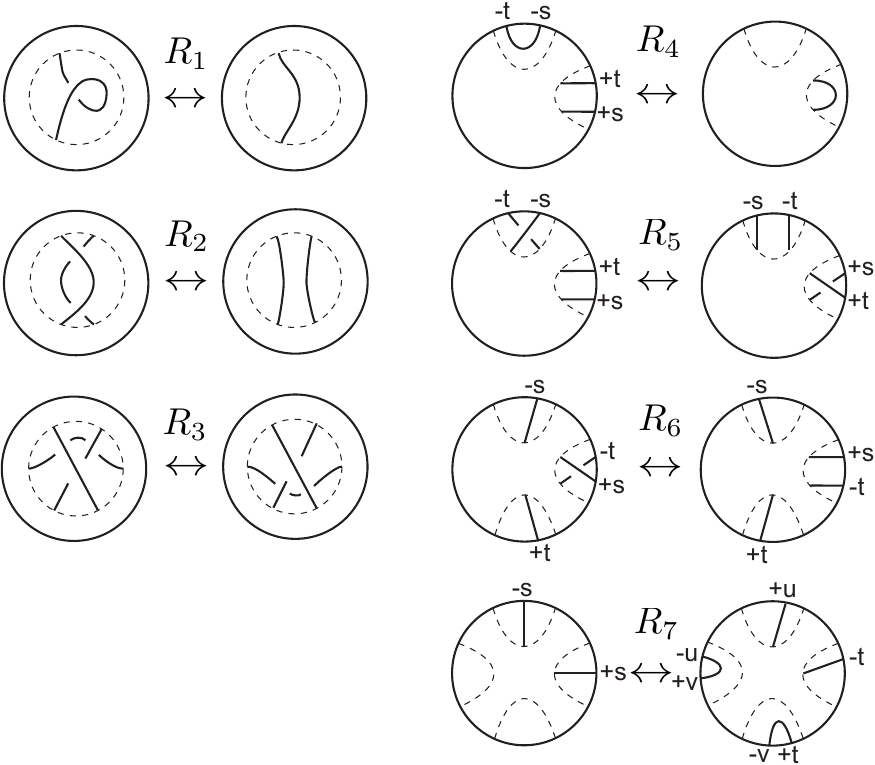}{Generalized Reidemeister moves for disk diagrams}

Manfredi~\cite{Manfredi14} gave a construction of a diagram of the lift of $L$ in $S^{3}$ from a standard disk diagram $D$ of $L$. 
The case of $L(4,1)$ is shown in Figure~\ref{fig:lift_4.pdf}. 
Here $D$ is a standard diagram 
in which the positive (resp. negative) boundary points are on the right (resp. upper) side. 
The diagram $\overline{D}$ is obtained from $D$ 
by the reflection about the axis from upper left to lower right 
and the inversion of the crossings. 
(The labels are preserved.) 
Following the notation in \cite{Manfredi14}, 
$\Delta_{n}^{-2}$ is the positive full twist of $n$ strands 
(the crossing is positive when the orientations of strands are upward). 

\fig[width=9cm]{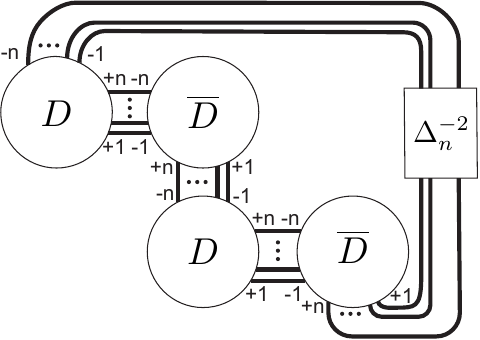}{The lift in $S^{3}$ of a link in $L(4,1)$}

For a link $L$ in $S^{3} / Q_{8}$, 
we consider the three lifts $L_{\bfi}$, $L_{\bfj}$, and $L_{\bfk}$ of $L$ in $L(4,1)$. 
Recall that the cube $C$ is a fundamental domain of $S^{3} / Q_{8}$. 
We decompose $C$ into four triangular prisms 
so that the projection in the direction of $\bfi$ induces the decomposition of a square by the diagonals. 
We construct $C_{\bfi}$ from $C$ and copies of the prisms 
by gluing the opposite faces using the counterclockwise $\pi / 2$-rotation 
as indicated in Figure~\ref{fig:square-cover.pdf}. 
Then $C_{\bfi}$ is a lens-shaped fundamental domain of the lens space $X_{\bfi}$. 
Since the dihedral angles of $C$ in the spherical metric are equal to $2\pi / 3$, 
the domain $C_{\bfi}$ is bounded by two orthogonal totally geodesic disks. 
After setting the link $L$ in general position for the prisms, 
we obtain the intersection of the lift $L_{\bfi}$ and the domain $C_{\bfi}$ 
by the construction of $C_{\bfi}$. 
Then we obtain a diagram of the lift $L_{\bfi}$. 
By changing the direction of projection, 
we also obtain diagrams of the lifts $L_{\bfj}$ and $L_{\bfk}$. 

\fig[width=12cm]{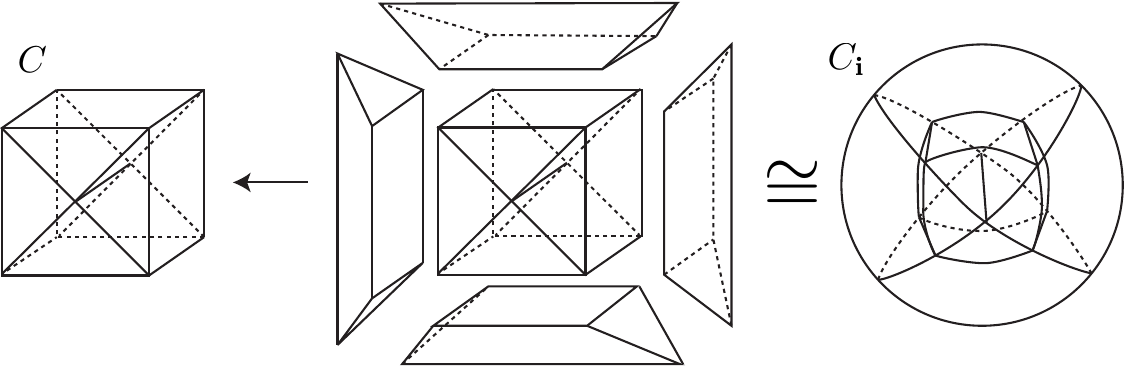}{Construction of a lens-shaped domain from the cube $C$}

Figure~\ref{fig:lift_square.pdf} shows an example of the construction of a lift. 
At first, we glue prisms containing links. 
By projecting to a disk, we obtain the diagram in the left of the second line. 
At this time, the boundary points in each outer prism are arranged on a radius. 
Keeping the condition that 
the boundary point with label $+t$ is mapped to that with $-t$ by the counterclockwise $\pi / 2$-rotation, 
we isotope this diagram 
so that any two boundary points are not placed on a common radius. 
Furthermore, we push out the boundary points to the boundary of the disk. 
When a new crossing occurs, 
a boundary point with a positive (resp. negative) label passes over (resp. under). 
Then we obtain the disk diagram in the right of the second line. 
We next obtain the standard disk diagram in the left of the last line by $R_{6}$ moves. 
We may adjust diagrams by other moves and relabeling as shown in the last line. 

\fig[width=12cm]{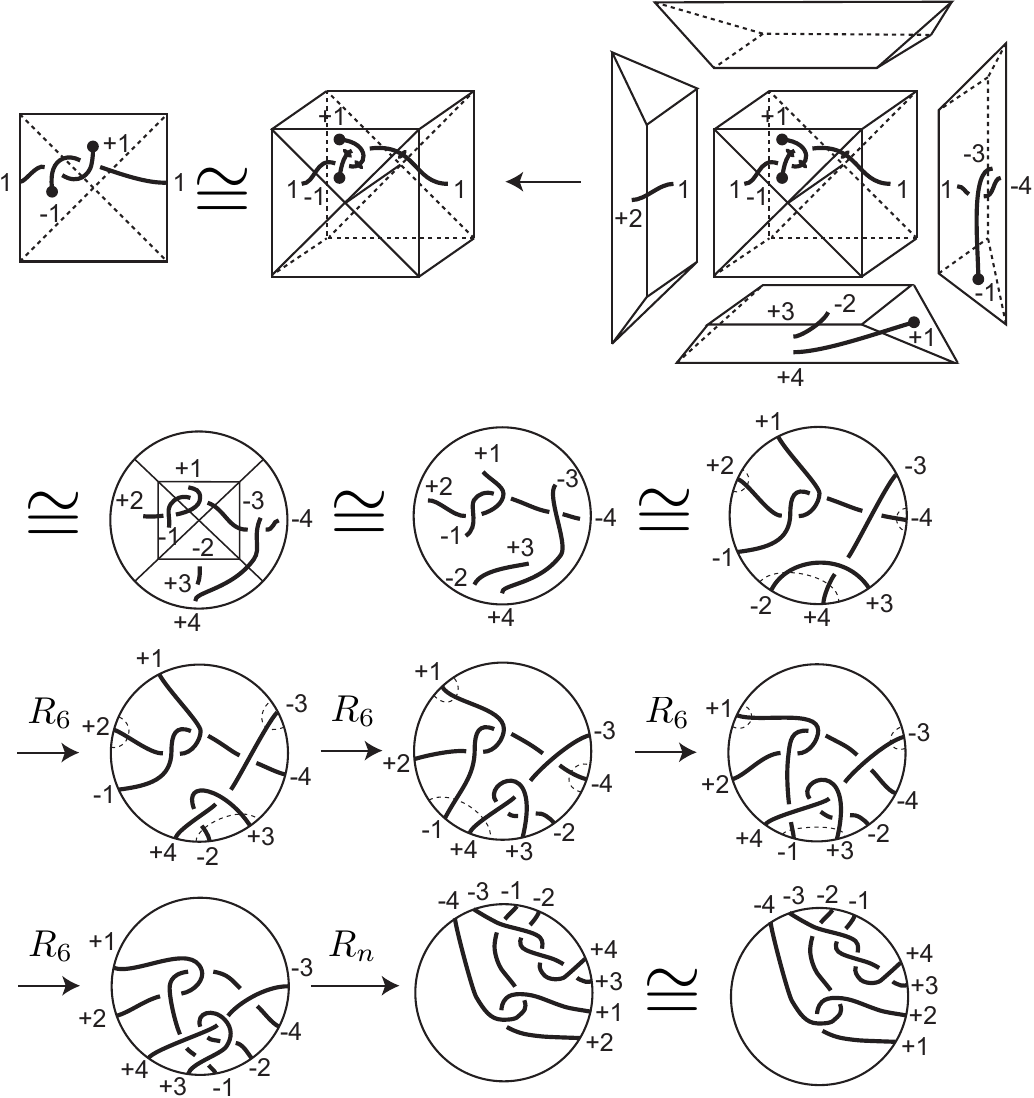}{A lift in $L(4,1)$ of a link in $S^{3} / Q_{8}$}

Note that identifications of $X_{\bfi}$, $X_{\bfj}$, and $X_{\bfk}$ with $L(4,1)$ is not unique. 
Indeed, the mapping class group of $L(4,1)$ is isomorphic to $\bbZ / 2\bbZ$, 
and its nontrivial element is obtained by reversing the lens-shaped fundamental domain 
(see \cite{CM18}). 
Nevertheless, the isotopy classes of the links $L_{\bfi}$, $L_{\bfj}$, and $L_{\bfk}$ are uniquely determined. 

\begin{prop}
\label{prop:reverse}
Fix a double covering map from $L(4,1)$ to $S^{3} / Q_{8}$. 
Let $L'$ be a link in $L(4,1)$ that is a lift of a link $L$ in $S^{3} / Q_{8}$. 
Suppose that $\tau \colon L(4,1) \to L(4,1)$ is a homeomorphism that represents the nontrivial mapping class. 
Then the link $\tau (L')$ is isotopic to $L$. 
\end{prop}
\begin{proof}
Let $\iota \colon L(4,1) \to L(4,1)$ denote the nontrivial element of the deck transformation. 
Consider how $\iota$ acts on $\pi_{1}(L(4,1)) \cong \bbZ / 4\bbZ$. 
We may assume that $\pi_{1}(L(4,1)) = \{ \pm 1, \pm \bfi \}$. 
The conjugate by an element of $Q_{8} \setminus \{ \pm 1, \pm \bfi \}$ determines the action of $\iota$. 
Since $-\bfj \bfi \bfj = -\bfi$, 
the action of $\iota$ on $\{ \pm 1, \pm \bfi \}$ is nontrivial. 
Hence $\iota$ is isotopic to $\tau$. 
Since the link $L'$ is invariant under the deck transformation, 
the link $\tau (L')$ is isotopic to $L$. 
\end{proof}

Proposition~\ref{prop:reverse} can be also shown by the fact that the $\pi$-rotation of the cube $C$ around an axis orthogonal to two faces induces a self-homeomorphism of $S^{3} / Q_{8}$ isotopic to the identity.

\section{Examples}
\label{section:ex}

\subsection{Fibers of Seifert fibrations}
\label{subsection:fiber}
We begin with the simplest example. 
Suppose that $L$ is the knot in $S^{3} / Q_{8}$ shown in the upper left of Figure~\ref{fig:hopf-lift.pdf}. 
The lifts $L_{\bfi}$ and $L_{\bfj}$ of $L$ in $L(4,1)$ are obtained as indicated in Figure~\ref{fig:hopf-lift.pdf}. 
They are not isotopic, since $L_{\bfi}$ has two components, but $L_{\bfj}$ has a single component. 
The lift $L_{\bfk}$ is isotopic to $L_{\bfj}$. 
The lift $\widetilde{L}$ in $S^{3}$ is the Hopf link. 
This is a reconstruction of an example in \cite{Manfredi14}. 
Moreover, $L$ is a fiber of a Seifert fibration on $S^{3} / Q_{8}$. 
To describe this more precisely, 
we classify the Seifert fibrations on $S^{3} / Q_{8}$. 

\fig[width=12cm]{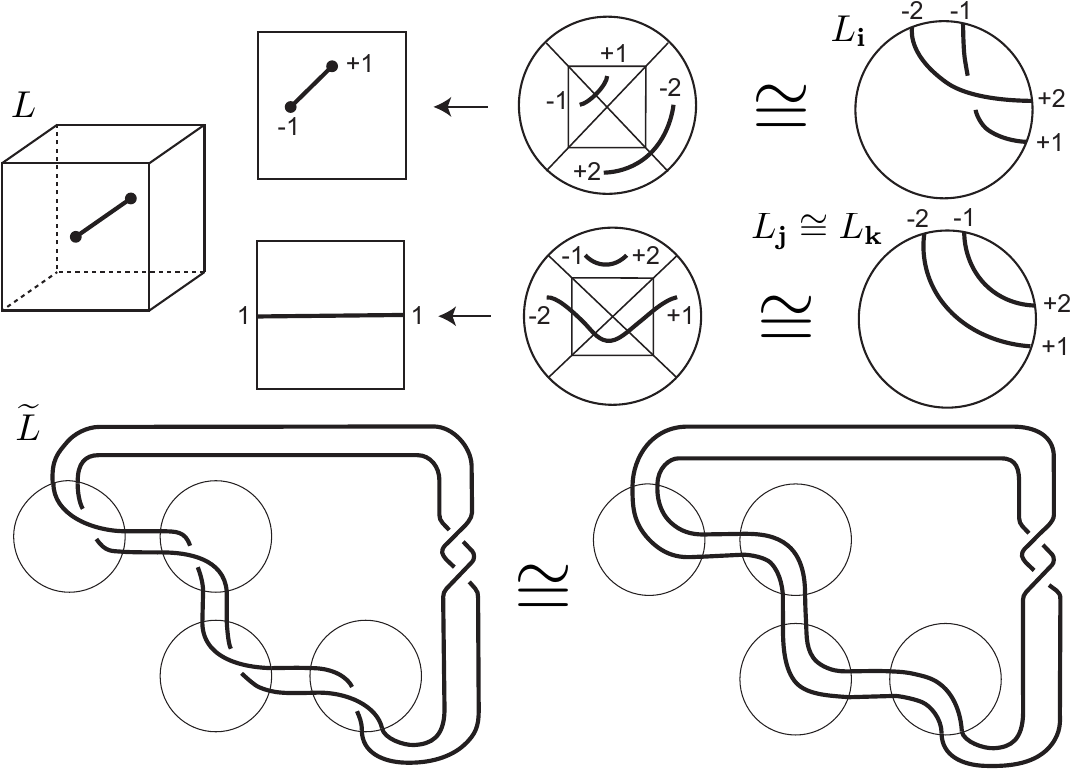}{Links in Example~\ref{subsection:fiber}}

\begin{prop}
\label{prop:seifert}
There are exactly four Seifert fibrations on the 3-manifold $S^{3} / Q_{8}$ up to isotopy. 
\end{prop}
\begin{proof}
It is known that a prism 3-manifold admits two Seifert fibrations up to homeomorphism 
(see \cite[Theorem 2.3]{Hatcher07}). 
More precisely,  the 3-manifold $S^{3} / Q_{8}$ admits the two Seifert fibrations 
that are $M(0, 0; 1/2, -1/2, 1/2)$ and $M(-1, 0; 2/1)$ in the notation of \cite{Hatcher07}. 
The former has three $(2,1)$-singular fibers, 
and its base orbifold is the sphere with three cone points of order two. 
The latter has no singular fibers, 
and it is a $S^{1}$-bundle over $\bbR \bbP^{2}$. 
The 3-manifold $S^{3} / Q_{8}$ can be decomposed into the twisted $I$-bundle $S^{1} \widetilde{\times} S^{1} \widetilde{\times} I$ over the Klein bottle and the solid torus. 
The two Seifert fibrations on $S^{3} / Q_{8}$ are obtained by extending those on $S^{1} \widetilde{\times} S^{1} \widetilde{\times} I$. 

We consider how the self-homeomorphisms of $S^{3} / Q_{8}$ act on the Seifert fibrations. 
Recall that $\pi_{0} (\Aut(S^{3} / Q_{8})) \cong \Out(Q_{8}) \cong \mathfrak{S}_{3}$. 
A Klein bottle $\Sigma$ is embedded in $S^{3} / Q_{8}$ as shown in Figure~\ref{fig:kb_fiber.pdf}. 
A tubular neighborhood $N(\Sigma)$ of $\Sigma$ is homeomorphic to $S^{1} \widetilde{\times} S^{1} \widetilde{\times} I$, and the remaining part is homeomorphic to the solid torus, 
whose core represents the free homotopy class $[\bfi]$. 

\fig[width=12cm]{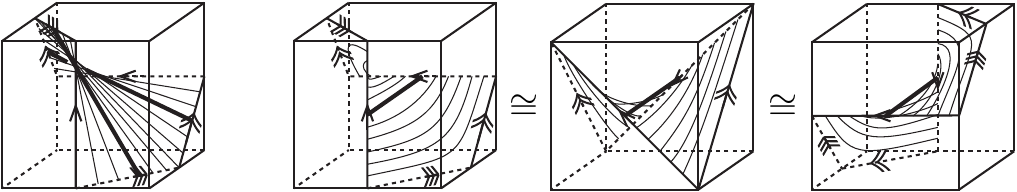}{A Klein bottle embedded in $S^{3} / Q_{8}$ and fibers}

The left of Figure~\ref{fig:kb_fiber.pdf} indicates a Seifert fibration on $N(\Sigma)$ with two singular fibers. 
The Seifert fibration on $S^{3} / Q_{8}$ obtained by extending this has three singular fibers 
representing the free homotopy classes $[\bfi]$, $[\bfj]$, and $[\bfk]$. 
This Seifert fibration has a symmetry that realizes each permutation of the singular fibers. 
Since each element of $\Out(Q_{8})$ is determined by the corresponding permutation of the classes $[\bfi]$, $[\bfj]$, and $[\bfk]$, 
each self-homeomorphism of $S^{3} / Q_{8}$ preserves this Seifert fibration up to isotopy. 

The right of Figure~\ref{fig:kb_fiber.pdf} indicates a Seifert fibration on $N(\Sigma)$ with no singular fibers. 
The Seifert fibration on $S^{3} / Q_{8}$ obtained by extending this is a $S^{1}$-bundle 
in which each fiber represents the free homotopy class $[\bfi]$. 
The $\pi/2$-rotation of $C$ around the axis $\bfi$ corresponds to the transposition of $[\bfj]$ and $[\bfk]$. 
This self-homeomorphism preserves the isotopy class of $\Sigma$ as shown in Figure~\ref{fig:kb_fiber.pdf}, 
and so it also preserves this Seifert fibration up to isotopy. 
Hence there are three Seifert fibrations on $S^{3} / Q_{8}$ with no singular fibers, 
determined by the free homotopy class $[\bfi]$, $[\bfj]$, or $[\bfk]$ represented by a fiber. 
Therefore there are four Seifert fibrations on $S^{3} / Q_{8}$ in total. 
\end{proof}

We denote the four Seifert fibrations on $S^{3} / Q_{8}$ by $\mathcal{F}_{-1}$, $\mathcal{F}_{\bfi}$, $\mathcal{F}_{\bfj}$, and $\mathcal{F}_{\bfk}$ 
based on the free homotopy class of a regular fiber. 
The knot $L$ in the first example can be represented by a singular fiber in $\mathcal{F}_{-1}$ or a fiber of $\mathcal{F}_{\bfi}$. 
Note that the lift of any Seifert fibration to a finite cover is also a Seifert fibration. 

The lift of $\mathcal{F}_{-1}$ to $X_{\bfi}$ is a Seifert fibration on $L(4,1)$ with two singular fibers, 
which are the lift $L_{\bfi}$. 
The lift of $\mathcal{F}_{-1}$ to $X_{\bfj}$ is the same Seifert fibration on $L(4,1)$, 
but the lift $L_{\bfj}$ is a regular fiber. 
The case for $X_{\bfk}$ is similar. 
According to the classification of the Seifert fibrations on $S^{3}$ \cite[Theorem 2.3]{Hatcher07}, 
a Seifert fibration on $S^{3}$ without singular fibers is the positive or negative Hopf fibration. 
A Hopf fibration is positive if and only if the linking number of its two fibers with coherent orientations is equal to $+1$. 
In the lift of $\mathcal{F}_{-1}$ to $X_{\bfi}$, 
the lifts of the singular fibers of $[\bfj]$ and $[\bfk]$ are regular fibers. 
In the lift $\widetilde{\mathcal{F}}_{-1}$ of $\mathcal{F}_{-1}$ to $S^{3}$, 
the lifts of all the singular fibers are regular fibers by the symmetry. 
Since $\widetilde{\mathcal{F}}_{-1}$ has no singular fibers, 
it is a Hopf fibration. 
The lift $\widetilde{L}$ in $S^{3}$ consists of two fibers of $\widetilde{\mathcal{F}}_{-1}$, 
and so it is the Hopf link. 
If we choose an orientation on the knot $L$ in Figure~\ref{fig:hopf-lift.pdf}, 
The orientation on the lift $\widetilde{L}$ is determined. 
Then the linking number of the two components of $\widetilde{L}$ is equal to $+1$. 
Hence $\widetilde{\mathcal{F}}_{-1}$ is the positive Hopf fibration. 
Note that the three lifts of a regular fiber of $\mathcal{F}_{-1}$ into $L(4,1)$ are isotopic. 
Since each self-homeomorphism of $S^{3} / Q_{8}$ preserves $\mathcal{F}_{-1}$ up to isotopy, 
the six diagrams of a regular fiber of $\mathcal{F}_{-1}$ depending on the choice of directions of $\bfi$, $\bfj$, and $\bfk$ are equivalent. 

The lift of $\mathcal{F}_{\bfi}$ to $X_{\bfi}$ is a $S^{1}$-bundle over $S^{2}$, 
and the lift $L_{\bfi}$ consists of two fibers. 
The lift of $\mathcal{F}_{\bfi}$ to $X_{\bfj}$ is a $S^{1}$-bundle over $\bbR \bbP^{2}$, 
and the lift $L_{\bfj}$ is a fiber. 
Since the lift $\widetilde{\mathcal{F}}_{\bfi}$ of $\mathcal{F}_{\bfi}$ to $S^{3}$ has no singular fibers 
and the lift $\widetilde{L}$ with a coherent orientation is the positive Hopf link, 
the lift $\widetilde{\mathcal{F}}_{\bfi}$ is also the positive Hopf fibration.

We can construct satellites of the knot $L$. 
We define the satellite construction in a general setting as follows. 
Let $K$ be an $n$-component link in a 3-manifold $M$. 
Let $K_{i}$ for $i = 1, \dots, n$ denote the components of $K$. 
Fix a close regular neighborhood $N(K) = N(K_{1}) \cup \dots \cup N(K_{n})$ of $K$. 
The meridian of each $K_{i}$ is a simple closed curve on $\partial N(K_{i})$ 
which bounds a disk in $N(K_{i})$. 
Fix an oriented longitude of each $K_{i}$ which is a simple closed curve on $\partial N(K_{i})$ 
and transversely intersects the meridian at a single point. 
Let $P = (P_{1}, \dots, P_{n})$ be an $n$-tuple of links in the solid torus $S^{1} \times D^{2}$. 
Suppose that each $P_{i}$ is not contained in a 3-ball. 
Then the \emph{satellite} $K(P)$ of $K$ with the pattern $P$ is 
a link in $M$ obtained as the image of $P_{1} \cup \dots \cup P_{n}$ 
by replacing each $N(K_{i})$ with $S^{1} \times D^{2}$ 
so that the oriented longitude $S^{1} \times \{ z \} \subset \partial (S^{1} \times D^{2})$ for a point $z \in \partial D^{2}$ is mapped to the oriented longitude of $K_{i}$. 
For the knot $L$ in Figure~\ref{fig:hopf-lift.pdf} and a pattern $P$, 
the lifts $L(P)_{\bfi}$ and $L(P)_{\bfj}$ in $L(4,1)$ of the satellite link $L(P)$ in $S^{3} / Q_{8}$ may be non-isotopic. 
Their lift $\widetilde{L(P)} = \widetilde{L}(\widetilde{P})$ in $S^{3}$ is the satellite of the Hopf link $\widetilde{L}$ with the pattern $\widetilde{P}$ that consists of two copies of the lift of $P$ in the double cover of $S^{1} \times D^{2}$. 
This reconstructs other examples in \cite{Manfredi14}. 
Since the Hopf link and its satellites are not hyperbolic, 
all of the links in Example~\ref{subsection:fiber} are not hyperbolic.

\subsection{Hyperbolic links}
\label{subsection:hyp}

We give an example of a hyperbolic knot in $S^{3} / Q_{8}$, which looks the simplest. 
Suppose that $L$ is the knot in $S^{3} / Q_{8}$ shown in Figure~\ref{fig:hyp-lift.pdf}. 
The lifts $L_{\bfi}$ and $L_{\bfj}$ in $L(4,1)$ and the lift $\widetilde{L}$ in $S^{3}$ are obtained as indicated in Figure~\ref{fig:hyp-lift.pdf}. 
The lift $L_{\bfj}$ is the same as the example in Figure~\ref{fig:lift_square.pdf}, but it is relabeled first. 
The lifts $L_{\bfi}$ and $L_{\bfj}$ are not isotopic, since $L_{\bfi}$ has two components, but $L_{\bfj}$ has a single component. 
The lift $L_{\bfk}$ is isotopic to $L_{\bfj}$. 
The link $\widetilde{L}$ in $S^{3}$ is shown by two diagrams in Figure~\ref{fig:hyp-lift.pdf}. 
Although not obvious, they are isotopic. 

\fig[width=12cm]{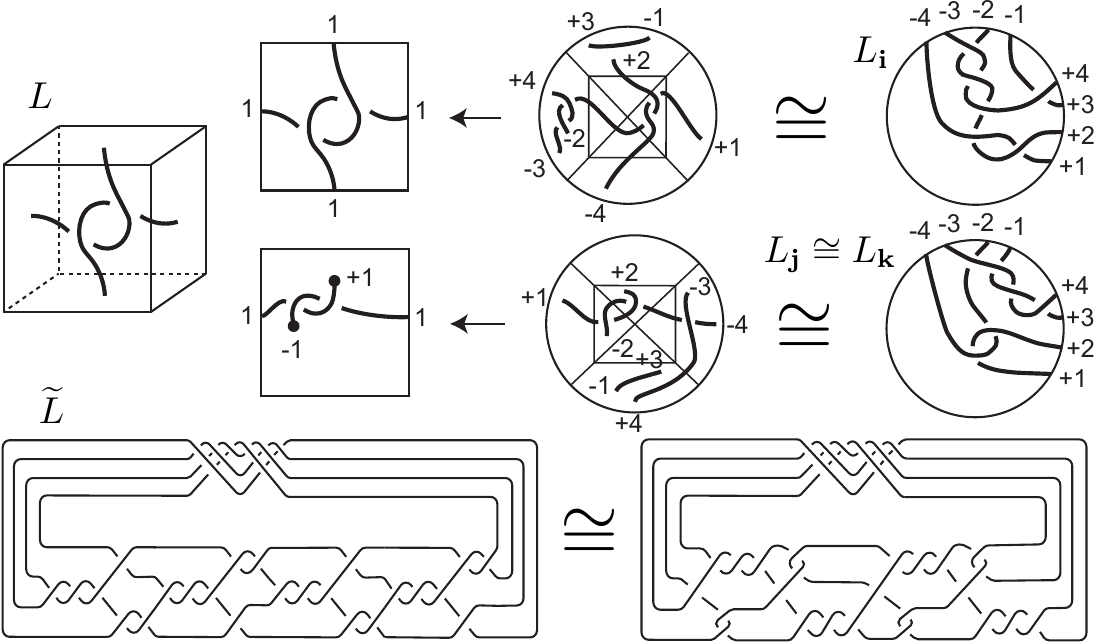}{Links in Example~\ref{subsection:hyp}}

We show that the complement $M = (S^{3} / Q_{8}) \setminus L$ is s119 in the census \cite{CHW99} of cusped finite-volume hyperbolic 3-manifolds. 
The manifold $M$ can be decomposed into a certain ideal polyhedron $P$ with eight faces as indicated in Figure~\ref{fig:s119.pdf}. 
The manifold $M$ is obtained from $P$ by gluing the mutually opposite faces so that the arrows are matched. 
The link $L$ is preserved by a $\bbZ / 2 \bbZ \times \bbZ / 2 \bbZ$-action on $S^{3} / Q_{8}$ 
whose non-trivial elements are induced by the three $\pi$-rotations of the cube $C$ preserving the diagram of $L$. 
This induces a $\bbZ / 2 \bbZ \times \bbZ / 2 \bbZ$-symmetry of the manifold $M$. 
Note that the isometry group of the hyperbolic 3-manifold s119 is isomorphic to $\bbZ / 2 \bbZ \times \bbZ / 2 \bbZ$ according to SnapPy \cite{SnapPy}. 
The decomposition of $M$ into $P$ is compatible with this $\bbZ / 2 \bbZ \times \bbZ / 2 \bbZ$-symmetry. 
The polyhedron $P$ can be decomposed into six tetrahedra compatibly with the gluing, 
but the symmetry is broken. 
It may be possible to calculate the explicit parameters of ideal tetrahedra based on the way of gluing, 
and identify $M$. 
However, we check the identification more easily. 

The double cover $M_{\bfi} = L(4,1) \setminus L_{\bfi}$ of $M$ is homeomorphic to the complement of a link in $S^{3}$ as indicated in Figure~\ref{fig:otet8_5.pdf}. 
In the second diagram, the complement of the core (denoted by 2) in the lens-shaped region is decomposed into four prisms. 
The equivalences between the second and fifth diagrams are obtained by regluing the four prisms. 
In the fifth diagram, the opposite disks are glued using the counterclockwise $\pi /2$-rotation. 
According to the census \cite{FGGTV16} of hyperbolic 3-manifolds which can be decomposed into regular ideal tetrahedra (called \emph{tetrahedral manifolds}), 
the manifold $M_{\bfi}$ is otet$08_{0005}$ 
(see Figure 3 in \cite{FGGTV16}), also named t12844 in the census \cite{CHW99}. 
The fact that the complement of the link shown in the last diagram is otet$08_{0005}$ 
can be confirmed by SnapPy \cite{SnapPy}. 
The manifold $M_{\bfi}$ can be decomposed into eight regular ideal tetrahedra. 
Any tetrahedral manifold is commensurable to the figure-eight knot complement m004 
and the Bianchi orbifold $\bbH^{3} / \mathrm{PSL}(2, \bbZ [(1+\sqrt{-3})/2])$, 
and so it is arithmetic (see \cite[Section 5]{FGGTV16}). 
Hence $M_{\bfi}$ and $M$ are arithmetic hyperbolic 3-manifolds. 
Since $\vol(M_{\bfi}) = 8 v_{\mathrm{tet}}$, 
we have $\vol(M) = 4 v_{\mathrm{tet}}$, where $v_{\mathrm{tet}} = 1.0149...$ is the volume of a regular ideal tetrahedron. 
Baker and Reid \cite[Table 1]{BR02} computed which finite groups are the fundamental groups of Dehn fillings of 1-cusped arithmetic hyperbolic 3-manifolds decomposed into at most 7 tetrahedra. 
Among them, only the manifold s119 has $S^{3} / Q_{8}$ as a Dehn filling. 
Hence $M$ is s119. 
Note that $M$ cannot be decomposed into regular ideal tetrahedra, 
since any ideal triangulation of s119 needs at least six ideal tetrahedra 
according to the census \cite{CHW99}. 
The link $\widetilde{L}$ in $S^{3}$ and its complement are not in any census at present 
because they are too large. 

\fig[width=12cm]{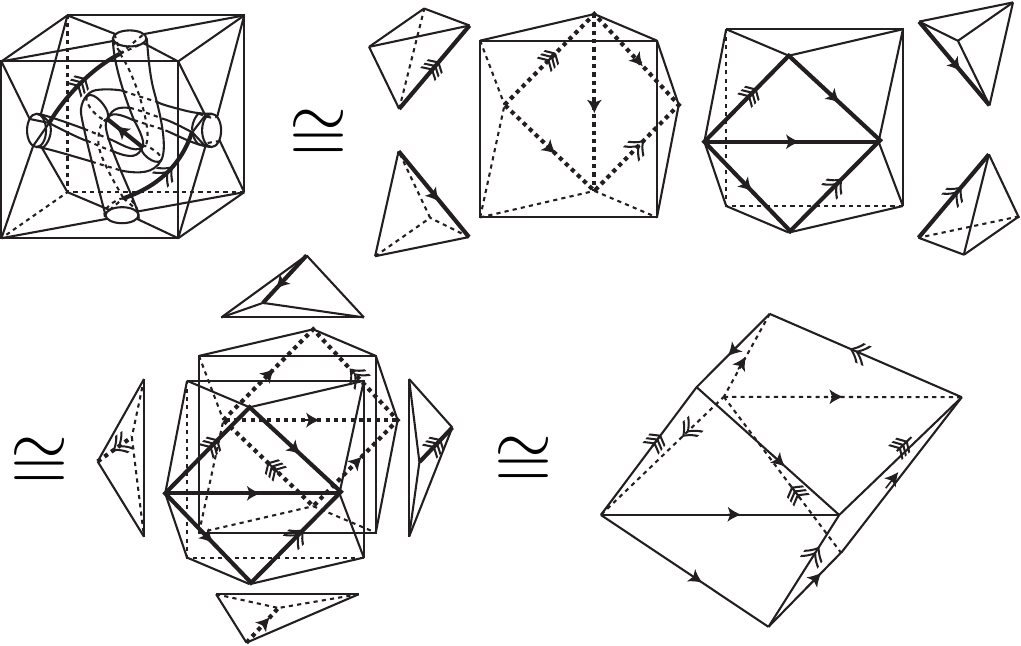}{A polyhedral decomposition of the manifold s119}

\fig[width=12cm]{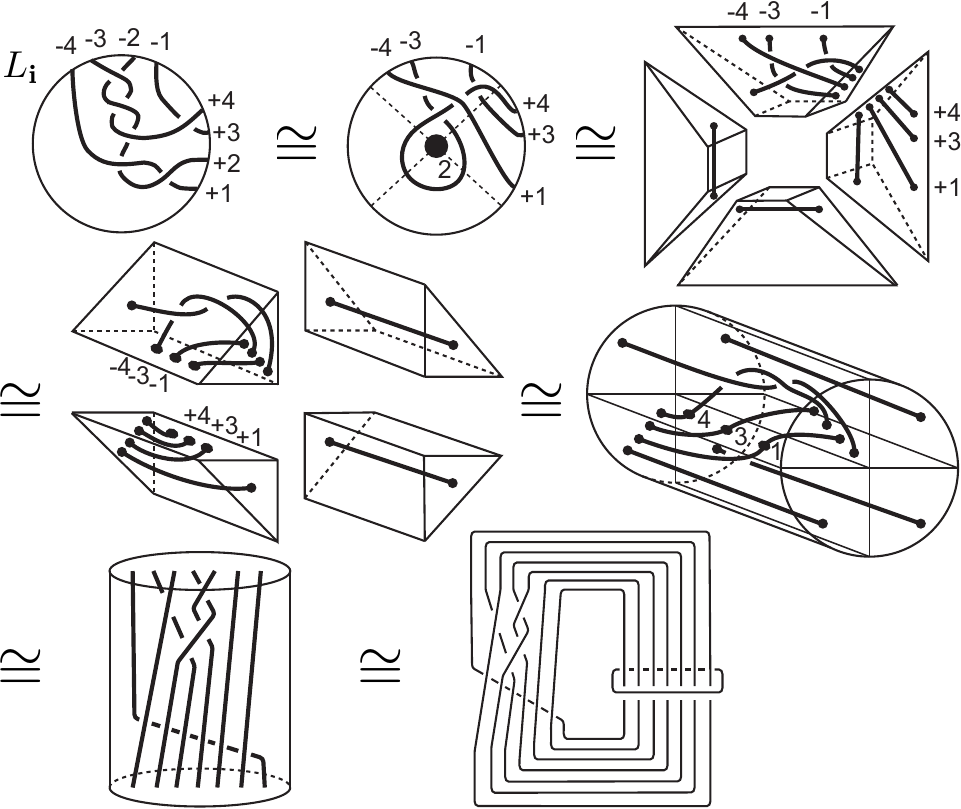}{Links whose complements are the manifold otet$08_{0005}$}

Remark that the knot with the converse crossings is not hyperbolic. 
Figure~\ref{fig:converse-cross.pdf} indicates moves for this knot 
and a satellite construction from the knot in Example~\ref{subsection:fiber}. 

\fig[width=12cm]{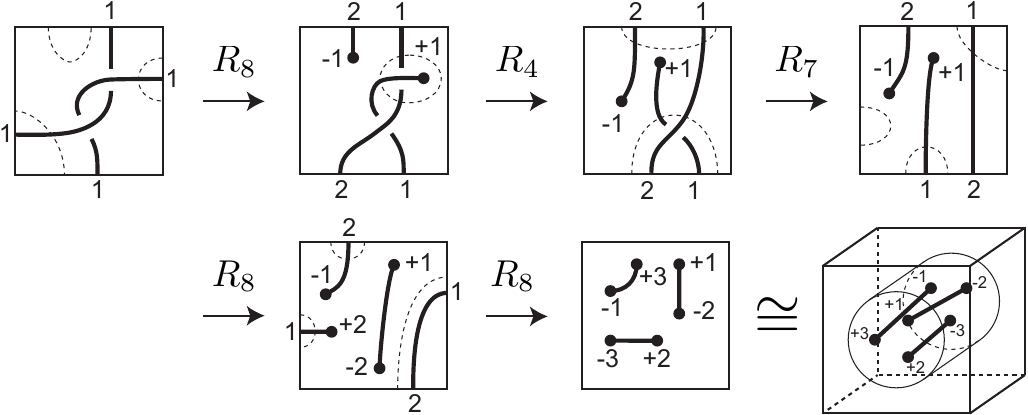}{Moves for a knot in $S^{3} / Q_{8}$}

\subsection{Links of the same number of components}
\label{subsection:2-comp}

Finally, we give an example of a hyperbolic link in $S^{3} / Q_{8}$ 
whose three lifts in $L(4,1)$ are not isotopic and have the same number of components. 
Suppose that $L$ is the knot in $S^{3} / Q_{8}$ shown in Figure~\ref{fig:hyp-lift_nh.pdf}. 
The lifts $L_{\bfi}$, $L_{\bfj}$, and $L_{\bfk}$ in $L(4,1)$ are obtained as indicated in Figure~\ref{fig:hyp-lift_nh.pdf}. 
(We omit diagrams of the lift $\widetilde{L}$ in $S^{3}$.) 
The links $L_{\bfi}$, $L_{\bfj}$, and $L_{\bfk}$ have two components. 
Consider the lift in $S^{3}$ of a single component of each of $L_{\bfi}$, $L_{\bfj}$, and $L_{\bfk}$. 
Note that the lift in $S^{3}$ does not depend on the choice of a component. 
They are the 4-component links shown in Figure~\ref{fig:4-chain.pdf}, which are distinguished by splitness and hyperbolicity. 
Hence $L_{\bfi}$, $L_{\bfj}$, and $L_{\bfk}$ are not isotopic. 
The hyperbolicity of $\widetilde{L}$, which implies that of $L$, is verified by SnapPy \cite{SnapPy}. 

\fig[width=12cm]{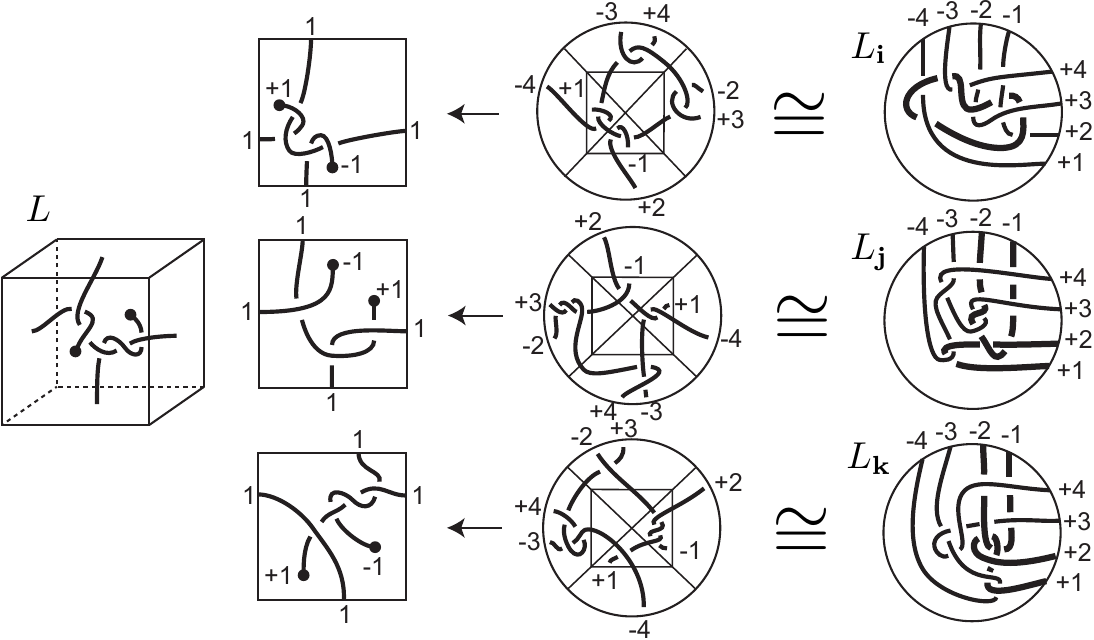}{Links in Example~\ref{subsection:2-comp}}

\fig[width=12cm]{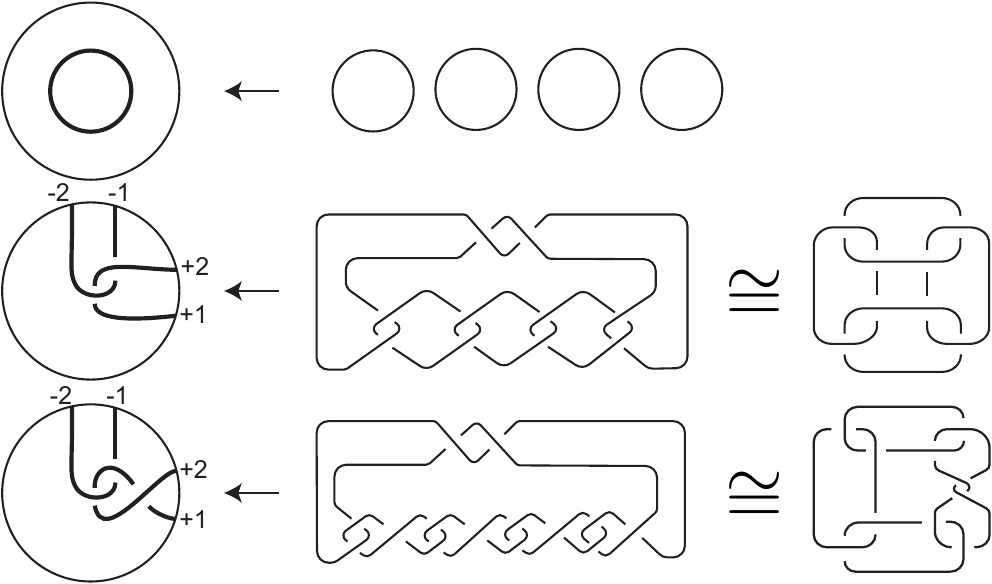}{The lifts in $S^{3}$ of components of $L_{\bfi}$, $L_{\bfj}$, and $L_{\bfk}$}

\section*{Acknowledgements} 
The author is grateful to Yuya Koda, Yuka Kotorii, Sonia Mahmoudi, Elisabetta Matsumoto, and Yuta Nozaki for their helpful discussions. 
This work is supported by the World Premier International Research Center Initiative Program, International Institute for Sustainability with Knotted Chiral Meta Matter (WPI-SKCM$^2$), MEXT, Japan.

\bibliographystyle{plain}
\bibliography{ref-lqg}

\end{document}